\title{Integration of first-order ODEs by Jacobi fields} 
\author{A. J. Pan-Collantes\thanks{
  Department of Mathematics, IES San Juan de Dios, Medina Sidonia, Spain; email: antonio.pan@uca.es}
        \and
        J. A. Alvarez-Garcia \thanks{
          Department of Mathematics, IES Jorge Juan, San Fernando, Spain; email: jose.alvg@gmail.com
                  }\thanks{}
        }
\documentclass{article}
\usepackage{graphicx}
\usepackage{tikz-cd}
\usepackage[utf8]{inputenc}  
\usepackage{amssymb, amsmath}
\usepackage{amsthm}
\usepackage{hyperref}
\hypersetup{
  colorlinks=true,
}
\hypersetup{
    linkcolor=magenta,
    filecolor=magenta,      
    urlcolor=cyan,
    citecolor=magenta,
    }
\usepackage{enumitem}
\usepackage[textwidth=3cm]{todonotes}
\usepackage{xcolor}

\theoremstyle{remark}

\theoremstyle{definition}

\newtheorem{proposition}{Proposition}[section]
\newtheorem{definition}{Definition}[section]
\newtheorem{example}{Example}[section]

\begin{document}

\newpage
\maketitle
\begin{abstract}
  A new class of vector fields enabling the integration of first-order ordinary differential equations (ODEs) is introduced. These vector fields are not, in general, Lie point symmetries. The results are based on a relation between 2-dimensional Riemannian manifolds and the integrability of first-order ODEs, which was established in a previous work of the authors. An integration procedure is provided, together with several examples to illustrate it. A connection between integrating factors of first-order ODEs and Schr\"odinger-type equations is highlighted.
\end{abstract}





\section{Introduction}

The classical Lie method for integrating first-order ordinary differential equations (ODEs) of the form $u'(x)=\phi(x,u)$ relies on identifying a vector field $X$ which is the infinitesimal generator of a non-trivial one-parameter group of symmetries \cite{ibragimovlibro,stephani,olver86}. These vector fields are characterized by the following condition:
\begin{equation}\label{Liesyme}
  [X,A]=\rho A,
\end{equation}
where $[,]$ denotes the Lie bracket and $A$ is the vector field associated to the ODE. Once this infinitesimal generator $X=\xi \partial_x+\eta \partial_u$ is known, the integration of the ODE is straightforward, since $(\eta-\xi \phi)^{-1}$ is an integrating factor \cite{sherring1992geometric,pancinf-struct}.

In the recent preprint \cite{pancollantes2023surfaces}, the authors have introduced a theory that connects surfaces, in the sense of 2-dimensional Riemannian manifolds, with the integration of first-order ODEs. In particular, it is demonstrated that there exist vector fields with a clear geometric meaning, and distinct from Lie point symmetries, which enable the integration of the given ODE. These vector fields arise as a modification of the notion of Jacobi field, which plays a key role in the study of Riemannian geometry, particularly in understanding the behavior of geodesics and its relation with the Gaussian curvature. 

In this paper, we introduce the concept of relative Jacobi field in Section \ref{preliminaries}, following an overview of the background and notation. Subsequently, in Section \ref{sec_integration}, we present a procedure for integrating first-order ODEs using a relative Jacobi field. We explore specific scenarios, such as surfaces with constant curvature, where the knowledge of a relative Jacobi field is obtained automatically; and surfaces with curvature dependent solely on the variable $x$, where the integration can be done in terms of a particular solution to a Schr\"odinger-type equation. Finally, in Section \ref{sec_examples} we present several illustrative examples.

\section{Relative Jacobi fields}\label{preliminaries}

In this work, we investigate the local integration of first-order ODEs of the form
\begin{equation}\label{odemaestra}
  u'(x)=\phi(x,u),
\end{equation}
where $\phi$ is a smooth function defined on an open set $U\subseteq \mathbb R^2$ which can be shrunk as required. Recall that the integration of \eqref{odemaestra} is equivalent to the finding of a first integral of the Pfaffian equation
\begin{equation}\label{pfaffianeq}
  -\phi dx+du\equiv 0,
\end{equation}
i.e., a smooth function $F\in\mathcal C^{\infty}(U)$ such that $dF=\mu(-\phi dx+du)$, for certain non-vanishing function $\mu\in\mathcal C^{\infty}(U)$ called an integrating factor for \eqref{pfaffianeq}.

Following the ideas developed in \cite{pancollantes2023surfaces}, we define in $U$ a Riemannian metric $g$ expressed in matrix form as
\begin{equation}
  \begin{pmatrix}
  1+ \phi^2 & -\phi \\
  - \phi  & 1
  \end{pmatrix}.
\end{equation}
The pair $\mathcal S=(U,g)$ defines a 2-dimensional Riemannian manifold, hereafter referred to as the \emph{surface associated with} equation \eqref{odemaestra}. 

The vector field defined on $U$ given by $A=\partial_x+\phi\partial_u$ is frequently called the associated vector field to equation \eqref{odemaestra}, and its integral curves are in correspondence with solutions to \eqref{odemaestra}. It can be checked that $g(A,A)=g(\partial_u,\partial_u)=1$ and $g(A,\partial_u)=0$, so the vector fields $A$ and $\partial_u$ constitute an orthonormal frame for the tangent bundle $TU$ \cite{pancollantes2023surfaces,Bayrakdar2018a}. 

On the surface $\mathcal S$ we consider the corresponding Levi-Civita connection $\nabla$, which is the unique torsion-free connection defined on the tangent bundle compatible with $g$ \cite{chen1999lectures,Morita,docarmoriemannian}. Given two vector fields $X,Y$ on $U$, they can be expressed in the frame $(A,\partial_u)$ as
$$
\begin{aligned}
  X&=x_1A+x_2\partial_u, \\
  Y&=y_1A+y_2\partial_u,
\end{aligned}
$$
where $x_i,y_i\in \mathcal C^{\infty}(U)$, $i=1,2$. We have that (see \cite{pancollantes2023surfaces} and references therein)
$$
\nabla_X Y=(X(y_1)-\phi_u x_2 y_2)A+(X(y_2)+\phi_u x_2 y_1)\partial_u.
$$

It is well-known that an important notion in the study of Riemannian \mbox{manifolds} is the concept of geodesic. Moreover, the study of geodesic vector fields, i.e., vector fields whose integral curves are geodesics, have attracted increasing attention in recent research \cite{Carinena2023,deshmukh19,deshmukh20}. They are characterized as follows.
\begin{definition}
  A vector field $X$ will be called a geodesic vector field if 
  $$
  \nabla_{X}X=0.
  $$
\end{definition}
Remarkably, note that in the surface $\mathcal S$ associated to equation \eqref{odemaestra},
$$
\nabla_A A=(A(1)-\phi_u\cdot 0\cdot 0)A+(A(0)+\phi_u \cdot 0\cdot 1)\partial_u=0,
$$
so $A$ is a geodesic vector field. 

The Gaussian curvature of a surface is a pivotal notion, arising in multiple mathematical and physical contexts. In the case of the surface $\mathcal{S}$ associated to equation \eqref{odemaestra}, Gaussian curvature can be calculated with the following simple formula \cite{Bayrakdar2018a,pancollantes2023surfaces}:
\begin{equation}\label{curvatureK}
  \mathcal{K}(x,u)=-\partial_u(A(\phi(x,u))).  
\end{equation}

Gaussian curvature and geodesics are related by the notion of Jacobi vector fields along geodesic curves \cite{docarmoriemannian,lee2018introduction}, which represent the infinitesimal separation between nearby geodesics. They can be thought of as measuring how geodesics spread out or come together as they move through the space. In \cite{pancollantes2023surfaces} it was introduced the following notion, deeply related to Jacobi vector fields along geodesics:
\begin{definition}
  Given a geodesic vector field $X$, a vector field $J$  will be called a Jacobi vector field relative to $X$ if the following equation holds
  \begin{equation}\label{JacobiEqGen}
    \nabla_{X} \nabla_{X} J+\mathcal{K}\left(g(X,X) J-g(J,X)X\right)=0.
    \end{equation}
\end{definition}

We now introduce two linear operators: given $h\in \mathcal C^{\infty}(U)$ we define:
$$
\begin{aligned}
  &\mathfrak S(h):=A(h)-\phi_u h,\\
  &\mathfrak T(h):=A(h)+\phi_u h.\\
\end{aligned}
  $$
These operators satisfy the following properties:
\begin{proposition}\label{operadores}
Consider the first-order ODE \eqref{odemaestra} with the corresponding Pfaffian equation \eqref{pfaffianeq}. Given a non-vanishing smooth function $h\in \mathcal C^{\infty}(U)$:
\begin{enumerate}[label=\alph*)]
  \item\label{intoperator} If $\mathfrak T (h)=0$ then $h$ is an integrating factor for \eqref{pfaffianeq}.

  \item\label{symoperator} If $\mathfrak S (h)=0$ then $h^{-1}$ is an integrating factor for \eqref{pfaffianeq}.

  \item\label{descompos} $(\mathfrak T \circ \mathfrak S)(h)=(A^2+\mathcal K)(h)$, where $\mathcal K$ is the curvature of the associated surface $\mathcal S$.
\end{enumerate}  
\end{proposition}

\begin{proof}
  To prove part a) observe that, by Poincaré Lemma, a smooth function $h$ is an integrating factor for \eqref{pfaffianeq} if and only if
  $$
  d(-\phi h dx+h du)=(\phi_u h+\phi h_u +h_x)dx\wedge du=0,
  $$
  and therefore if and only if 
  $$
  \phi_u h+\phi h_u +h_x=(A+\phi_u) (h)=\mathfrak T (h)=0.
  $$ 
  
  For part b), note that
  $$
  d(-\phi h^{-1} dx+h^{-1} du)=\left(\frac{\phi_u}{ h}-\frac{\phi h_u}{h^2} -\frac{h_x}{h^2}\right)dx\wedge du.
  $$
  Then, $h^{-1}$ is an integrating factor if and only if 
  $$
\phi_u h-\phi h_u-h_x=-\mathfrak S (h)=0.
  $$
  
  Finally, for part c),
  $$
\begin{aligned}
  (\mathfrak T \circ \mathfrak S)(h)&=\mathfrak T(\mathfrak S(h))\\
  &=\mathfrak T(A(h)-\phi_u h)\\
  &=A(A(h)-\phi_u h)+\phi_u(A(h)-\phi_u h)\\
  &=A^2(h)-A(\phi_u h)+\phi_u A(h)-\phi_u^2 h\\ 
  &=A^2(h)-A(\phi_u) h-\phi_u A(h)+\phi_u A(h)-\phi_u^2 h\\
  &=A^2(h)-(A(\phi_u) +\phi_u^2) h\\
  &=A^2(h)-(\phi_{ux}+\phi \phi_{uu} +\phi_u^2) h\\
  &=A^2(h)-\partial_u(\phi_x+\phi \phi_u)h\\
  &=(A^2+\mathcal K)(h),
\end{aligned}
  $$
  where we have used equation \eqref{curvatureK}.
\end{proof}

In the next section we will show the relevance of these operators in the integration of a first-order ODE like \eqref{odemaestra}, once a Jacobi field relative to $A$ is known.

\section{Integration of first-order ODEs}\label{sec_integration}
Given a first-order ODE of the form \eqref{odemaestra}, the associated vector field $A$ is a geodesic vector field with $g(A,A)=1$, and therefore a Jacobi field $J$ relative to $A$ satisfies
\begin{equation}\label{JacobiEqA}
  \nabla_{A} \nabla_{A} J+\mathcal{K}\left(J-g(J,A)A\right)=0
\end{equation}
Observe that the vector fields of the form $J=(ax+b)A$, where $a,b\in \mathbb R$, trivially satisfy equation \eqref{JacobiEqA}, and thus we will refer to them as trivial Jacobi fields relative to $A$. 

Given a non-trivial Jacobi field $J$ relative to $A$, we can express it in the frame $\{A,\partial_u\}$ as
$$
J=\sigma A+\delta \partial_u,
$$
for certain smooth functions $\sigma,\delta \in \mathcal C^{\infty}(U)$.
By substitution into equation \eqref{JacobiEqA} we obtain
$$
\nabla_{A} \nabla_{A} \left(\sigma A+\delta \partial_u\right)+\mathcal{K}\delta \partial_u=0.
$$

Since $\nabla_A(A)=0$ and $\nabla_A(\partial_u)=0$ we have
$$
A^2(\sigma) A+A^2(\delta) \partial_u+\mathcal{K}\delta \partial_u=0,
$$
and thus,
\begin{align}
  A^2(\sigma) &= 0, \label{eq:sigma}\\
  A^2(\delta) +\mathcal{K}\delta &= 0. \label{eq:delta}
\end{align}

Thus, to integrate equation \eqref{odemaestra} using a non-trivial Jacobi field $J$ relative to $A$, one may follow the following steps:
\begin{enumerate}
  \item Decompose the vector field $J$ as
  $$
J=\sigma A+\delta \partial_u,
$$
for certain smooth functions $\sigma,\delta \in \mathcal C^{\infty}(U)$.

\item\label{FIF} Define $F:=A(\sigma)$, which according to equation \eqref{eq:sigma} is a first integral of the vector field $A$, and therefore a first integral of \eqref{pfaffianeq}. If $F$ is not a constant function, then equation \eqref{odemaestra} is solved.
 
\item \label{FIG} On the other hand, if $A(\sigma)=a$, where $a\in \mathbb R$, then define the function $G:=\sigma-ax$, which constitutes a first integral of equation \eqref{pfaffianeq}, since
  $$
A(G)=A(\sigma - ax)=A(\sigma)-a=0.
  $$

  If $G$ is not a constant function, the first-order ODE \eqref{odemaestra} is solved.  On the other hand, in the case where $G=\sigma-ax=b$ for some $b\in \mathbb{R}$, we have $\sigma=ax+b$. Consequently,  it is necessary that $\delta \neq 0$, as $\delta=0$ would imply that $J$ is a trivial Jacobi field relative to $A$.

\item \label{deltanozero} Assuming $\delta \neq 0$, and considering equation \eqref{eq:delta} and part \ref{descompos} in Proposition \ref{operadores}, the following relation holds
    \begin{equation}\label{compos0}
      A^2(\delta) +\mathcal{K}\delta=(\mathfrak{T} \circ \mathfrak{S})(\delta)=0.
    \end{equation}
    There are two possibilities:
    \begin{enumerate}
      \item If $\mathfrak{S}(\delta)=0$, then, as indicated in part \ref{symoperator} of Proposition \ref{operadores}, $\delta^{-1}$ serves as an integrating factor for \eqref{pfaffianeq}. Compute a smooth function $I=I(x,u)$ satisfying:
      $$
      dI=\delta^{-1}(-\phi dx+du)
      $$
      via quadrature, and equation \eqref{odemaestra} is solved.
      \item If $\mathfrak{S}(\delta)\neq 0$, part \ref{intoperator} of Proposition \ref{operadores} permits the use of $\mathfrak{S}(\delta)$ as an integrating factor for \eqref{pfaffianeq}, since $\mathfrak{T}(\mathfrak{S}(\delta))=0$ because of \eqref{compos0}. Then, the 1-form
      $$
      \mathfrak{S}(\delta)(-\phi dx+du)
      $$
      is closed, and a primitive $I=I(x,u)$ provides the solution to equation \eqref{odemaestra}.
    \end{enumerate}
  \end{enumerate}

In summary, \emph{the knowledge of a non-trivial Jacobi field relative to $A$} leads to the integration of the first-order ODE \eqref{odemaestra}.

\subsection{Constant curvature surfaces}\label{sec_constant}
If the curvature of the associated surface is constant, a non-trivial Jacobi field can always be found, and the ODE can be integrated. Suppose $\mathcal{K}=k$, with $k\in \mathbb R$. We can assume a Jacobi field relative to $A$ of the form $J=\delta(x)\partial_u$. In this case, equation \eqref{JacobiEqA} takes the particular form
$$
\nabla_{A} \nabla_{A} \delta(x)\partial_u+k\delta(x)\partial_u=0,
$$
and since $\nabla_A \partial_u=0$, the function $\delta$ must satisfy the second-order ODE
$$
\ddot{\delta}(x)+k\delta(x)=0.
$$
The general solution to this ODE is given by
\begin{equation}\label{eq:deltaKcte}
  \delta(x)=\begin{cases}
      A \cos{\sqrt{k}x}+B \sin{\sqrt{k}x}, & \text{ if } k>0,\\
      A+B x, & \text{ if } k=0,\\
      A \cosh{\sqrt{-k}x}+B \sinh{\sqrt{-k}x}, & \text{ if } k<0,
      \end{cases}
\end{equation}
with $A,B\in\mathbb{R}$, so we can choose a particular solution and proceed as in Step \ref{deltanozero} of the procedure described in Section \ref{sec_integration}.

\subsection{Relation to Schr\"odinger equation}
In the case where the curvature of the associated surface $\mathcal K=k(x)$ depends only on the variable $x$, a Jacobi field relative to $A$ could be found by solving the following Schr\"odinger-type equation:
\begin{equation}\label{SchrodingerTypeEq}
  \ddot{\delta}(x)+k(x)\delta(x)=0.
\end{equation}

Therefore, according to our procedure described in Section \ref{sec_integration}, an integrating factor for equation \eqref{odemaestra} could be expressed in terms of a particular solution of equation \eqref{SchrodingerTypeEq}. Recall that some specific values of $k(x)$ lead to solutions $\delta(x)$ expressed in terms of well-known special functions like Bessel, Mathieu, Airy or hypergeometric functions. In these cases, finding the integrating factor $\mathfrak{T}(\delta)$ provided by our method through a direct approach appears to be challenging.


Observe that in the general case, when $\mathcal K=\mathcal K(x,u)$ depends on both $x$ and $u$, our approach provides a relationship between integrating factors of first-order ODEs and the partial differential equation for $\delta=\delta(x,u)$:
\begin{equation}
  A^2(\delta)+\mathcal K \delta=0,
\end{equation}
which could be interpreted as a generalized Schr\"odinger equation.

\section{Examples}\label{sec_examples}
We will show in this section several examples to illustrate the previous results.
\begin{example}
Consider the first-order ODE given by 
\begin{equation}\label{odeej1}
  u'(x)=\frac{1}{u},
\end{equation} 
whose associated vector field is $A=\partial_x+\frac{1}{u} \partial_u$. The curvature corresponding to the surface associated to equation \eqref{odeej1} is, according to equation \eqref{curvatureK},
$$
\mathcal{K}=-\partial_u\left(A\left(\frac{1}{u}\right)\right)=-\frac{3}{u^4}.
$$

On the other hand, it can be checked that the vector field
$$
J=x \partial_x+\frac{x(u^2-x+1)}{u}\partial_u
$$
is a Jacobi field relative to $A$, since it satisfies equation \eqref{JacobiEqA}. We will use $J$ to integrate equation \eqref{odeej1}, following the Jacobi field-based method developed in Section \ref{sec_integration}. 

First, $J$ is expressed in the frame $\{A,\partial_u\}$ as
$$
J=x A+ \frac{x (u^2-x)}{u}\partial_u.
$$

Secondly, we note that the first integral presented in Step \ref{FIF}, $F:=A(x)=1$, is constant, as is the first integral described in Step \ref{FIG}, $G:=x-x=0$.

Thus, we proceed with Step \ref{deltanozero}, where $\delta=\frac{x (u^2-x)}{u}\neq 0$ satisfies
$$
(\mathfrak T \circ \mathfrak S)\left(\frac{x (u^2-x)}{u}\right)=0.
$$
Then, $\mathfrak S \left(\frac{x (u^2-x)}{u}\right)=u\neq 0$ must serve as an integrating factor for the Pfaffian equation ${-\frac{1}{u}dx+du}\equiv 0$. This is indeed the case, and it enables us to obtain, by quadrature, the first integral
$$
I(x,u)=-x+u^2.
$$
The general solution to \eqref{odeej1} is given by $-x+u^2=C$, where $C\in \mathbb R$.

It is interesting to observe that, according to the characterization given in equation \eqref{Liesyme}, $J$ is not a symmetry of equation \eqref{odeej1}, since:
$$
[J,A]=-\partial_x-\frac{u^2+1}{u}\partial_u.
$$
\end{example}

\begin{example}
  In this example, we consider the first-order ODE
  \begin{equation}\label{OdeKcte}
    u_1=e^x+\sqrt{2ue^{x}-e^{2x}  - u^2 + 1},
  \end{equation}
  with associated vector field given by
  $$
  A=\partial_x+\left(e^x+\sqrt{2ue^{x}-e^{2x}  - u^2 + 1}\right)\partial_u.
  $$
The curvature of the associated surface $\mathcal S$ is given by
  $$
\mathcal{K}=-\partial_u\left(A\left(e^x+\sqrt{2ue^{x}-e^{2x}  - u^2 + 1}\right)\right)=1.
$$

Since the curvature is constant and positive, according to equation \eqref{eq:deltaKcte} in Section \ref{sec_constant}, a Jacobi field relative to $A$ is given by
$$
J=\cos(x)\partial_u.
$$
It can be checked that this vector field is not a Lie point symmetry, by performing the corresponding commutation relation $[J,A]$.

Next, to integrate equation \eqref{OdeKcte} by means of $J$ with the procedure described in Section \ref{sec_integration}, we start at Step \ref{deltanozero}, since $\sigma=0$. In this case, $\delta=cos(x)$ and
$$
\mu:=\mathfrak S\left(\cos(x)\right)=-\sin(x)+\frac{cos(x)(u-e^x)}{\sqrt{2ue^{x}-e^{2x}  - u^2 + 1}}\neq 0.
$$
Then, $\mu$ is an integrating factor for the Pfaffian equation
$$
-\left(e^x+\sqrt{2ue^{x}-e^{2x}  - u^2 + 1}\right)dx+du\equiv 0.
$$ 
A first integral can therefore be obtained by quadrature,
$$
I(x,u)=-\arcsin(e^x-u)-x,
$$
so the general solution of equation \eqref{OdeKcte} is given in explicit form by 
$$
u(x)=e^x+\sin(x+C), \quad C\in \mathbb R.
$$
\end{example}

\begin{example}

Consider the class of first-order linear ODEs
\begin{equation}\label{odelineal}
  u'(x)=q(x)u+p(x),
\end{equation}
where $p,q$ are smooth real functions. The vector field associated with equation \eqref{odelineal} is given by
$$
A=\partial_x+(q(x)u+p(x))\partial_u.
$$
Note that the curvature of the associated surface depends only on $x$:
$$
\mathcal K=-\partial_u(A(\phi))= -q'(x) - q(x)^2.
$$
Then, to find a Jacobi field relative to $A$ we consider a particular solution of the following Schr\"odinger equation
\begin{equation}\label{SchrodingerEj}
  \ddot{\delta}(x)-\left(q'(x) + q(x)^2\right)\delta(x)=0,
\end{equation}
which allows us to construct the vector field $J=\delta(x) \partial_u$, which is orthogonal to $A$.


The reader can check that the solutions to equation \eqref{SchrodingerEj} are given by
$$
\delta(x) = C_1 \left( H(x) + C_2 \right) e^{Q(x)},
$$
where $Q$ is a smooth function satisfying $Q'(x)=q(x)$, $H$ is a smooth function satisfying $H'(x)=e^{-2Q(x)}$, and $C_1, C_2 \in \mathbb R$.

The choice $C_1=1,C_2=0$ provides the following Jacobi field relative to $A$:
\begin{equation}
  J=H(x)e^{Q(x)}   \partial_u,
\end{equation}
which can be used to integrate equation \eqref{odelineal}. We have $\sigma=0$ and $\delta=H(x)e^{Q(x)}$, so we continue in Step \ref{deltanozero}. We have that
$$
\mathfrak S\left(H(x)e^{Q(x)}\right)=A\left(H(x)e^{Q(x)}\right)-\phi_u H(x)e^{Q(x)}=e^{-Q(x)}\neq 0,
$$
and therefore, $\mu:=e^{-Q(x)}$ is an integrating factor for the Pfaffian equation
$$
-(q(x) u+p(x))dx+du \equiv 0.
$$ 
Indeed, the general solution to equation \eqref{odelineal} can be expressed as 
$$
-P(x)+u e^{-Q(x)}=C,
$$
where $C\in \mathbb R$ and $P$ is a smooth function satisfying $P'(x)=p(x)e^{-Q(x)}$.
\end{example}




\bibliographystyle{unsrt}  
\bibliography{references.bib}
\end{document}